\documentclass[12pt,a4paper,twoside,reqno]{amsart}
\allowdisplaybreaks

\usepackage[all,cmtip]{xy}
\usepackage{amsmath,amscd}
\usepackage{amssymb,amsfonts}
\usepackage[driver=pdftex,margin=3cm,heightrounded=true,centering]{geometry}
\usepackage{mathtools}
\usepackage{tensor}
\usepackage{url}
\usepackage[colorlinks=true,linkcolor=blue,citecolor=blue]{hyperref}
\usepackage{slashed}
\usepackage[new]{old-arrows}

\usepackage{graphicx}

\usepackage{a4wide}

\usepackage{thmtools}
\declaretheorem[style=theorem,name={Theorem}]{theoremletter}

\usepackage{amsthm}

\theoremstyle{plain}

\newtheorem{theorem}{Theorem}[section]
\newtheorem*{thm*}{Theorem}
\newtheorem{prop}[theorem]{Proposition}
\newtheorem{lemma}[theorem]{Lemma}

\theoremstyle{definition}
\newtheorem{dfn}[theorem]{Definition}

\theoremstyle{remark} 

\newtheorem{remark}[theorem]{Remark}

\theoremstyle{plain}

\usepackage{amscd}

\numberwithin{equation}{section}

\newcommand{\alpheqn}[1][\relax]{
     \refstepcounter{equation}
     \if#1\relax \relax
       \else \label{#1}
     \fi  
     \setcounter{saveeqn}{\value{equation}}%
    \setcounter{equation}{0}%
    \renewcommand{\theequation}{\thealphequation}}
\newcommand{\reseteqn}{\setcounter{equation}{\value{saveeqn}}%
     \renewcommand{\theequation}{\thearabicequation}}

\providecommand{\mathscr}{\mathcal} 
\IfFileExists{mathrsfs.sty}{
\usepackage{mathrsfs}}         
   {
     \IfFileExists{eucal.sty}{
        \usepackage[mathscr]{eucal}} 
   {
   }
}

\newcommand{\sa}{{\operatorname{sa}}}

\newcommand{\RR}{{\mathbb{R}}}

\newcommand{\vertiii}[1]{{\left\vert\kern-0.25ex\left\vert\kern-0.25ex\left\vert #1 
    \right\vert\kern-0.25ex\right\vert\kern-0.25ex\right\vert}}
\newcommand{\Bvert}[1]{{\Big\vert\kern-0.25ex\Big\vert\kern-0.25ex\Big\vert #1 
    \Big\vert\kern-0.25ex\Big\vert\kern-0.25ex\Big\vert}}
\newcommand{\bvert}[1]{{\big\vert\kern-0.25ex\big\vert\kern-0.25ex\big\vert #1 
    \big\vert\kern-0.25ex\big\vert\kern-0.25ex\big\vert}}
\newcommand{\nvert}[1]{{\vert\kern-0.25ex\vert\kern-0.25ex\vert #1 
    \vert\kern-0.25ex\vert\kern-0.25ex\vert}}

\renewcommand{\leq}{\leqslant}
\renewcommand{\geq}{\geqslant}

\newcommand{\cd}{\cdot}

\newcommand{\nn}{\mathbb{N}}

\newcommand{\rr}{\mathbb{R}}
\newcommand{\cc}{\mathbb{C}}

\newcommand{\al}{\alpha}

\newcommand{\de}{\delta}

\newcommand{\ep}{\varepsilon}

\newcommand{\C}[1]{\mathcal{#1}}

\newcommand{\T}[1]{\textup{#1}}

\newcommand{\B}[1]{\mathbb{#1}}

\newcommand{\q}{\qquad}

\newcommand{\dist}{{\operatorname{dist}}}
\newcommand{\dom}{{\operatorname{dom}}}
\newcommand{\distqgh}{{\operatorname{dist}^{\operatorname{Q}}_{\operatorname{GH}}}}
\newcommand{\distgh}{{\operatorname{dist}_{\operatorname{GH}}}}
\newcommand{\disth}{{\operatorname{dist}_{\operatorname{H}}}}
\newcommand{\sgn}{{\operatorname{sgn}}}

\makeatletter
\@namedef{subjclassname@2020}{%
  \textup{2020} Mathematics Subject Classification}
\makeatother

\begin{document}

\author{Jens Kaad}

\address{Jens Kaad, Department of Mathematics and Computer Science, University of Southern Denmark, Campusvej 55, DK-5230 Odense M, Denmark}
\email{kaad@imada.sdu.dk}

\author{David Kyed}
\address{David Kyed, Department of Mathematics and Computer Science, University of Southern Denmark, Campusvej 55, DK-5230 Odense M, Denmark}
\email{dkyed@imada.sdu.dk}

\subjclass[2020]{58B34, 46L89, 46L30} 

\keywords{Quantum metric spaces, Gromov-Hausdorff distance.}

\title{A comparison of two quantum distances}

\begin{abstract}
We show that Rieffel's quantum Gromov-Hausdorff distance between two compact quantum metric spaces is not equivalent to the ordinary Gromov-Hausdorff distance applied to the associated state spaces.
\end{abstract}

\maketitle

\section{Introduction}
Rieffel's theory of compact quantum metric spaces provides an elegant non-commutative extension of the theory of classical compact metric spaces, and part of its success is due to the fact that it allows for a natural analogue of the Gromov-Hausdorff distance \cite{gromov-groups-of-polynomial-growth-and-expanding-maps, edwards-GH-paper, Evans:Prob}. In fact, quite a few such notions have been proposed \cite{Rie:GHD, Ker:MQG, Li:CQG, Li:GH-dist, Lat:QGH, Lat:DGH, Lat:GPS, Lat:MGP, Kerr-Li}, each having certain advantages and disadvantages. The relationship between the different versions is generally well understood, but
quite recently Connes and van Suijlekom initiated a programme to re-cast the theory of non-commutative geometry in the setting of operator systems \cite{walter:GH-convergence, walter-connes:truncations, CvS:Truncations, GvS:TolFin},  in which a new variation appears: instead of measuring the quantum Gromov-Hausdorff distance between the quantum metric spaces in question, the convergence results in   \cite{walter:GH-convergence} are formulated by applying the classical Gromov-Hausdorff distance to the associated state spaces. 
This emphasises the question whether this is only a minor cosmetic  difference or if these two quantum metrics are actually in-equivalent, and the aim of the present note is to show that the latter  situation is the case.  \\

To be more precise, if $(X,L_X)$  is a compact quantum metric space (see Section \ref{sec:qcms} for definitions), the associated state space $\C S(X)$ naturally becomes a compact metric space for the Connes metric
$
d_{L_X}(\mu, \nu):=\left\{ |\mu(x)-\nu(x)| : L_X(x)\leq 1   \right\}.
$  
Given two compact quantum metric spaces, $(X,L_X)$ and $(Y,L_Y)$, one may therefore either consider their quantum Gromov-Hausdorff distance $\dist_{\text{GH}}^{\text{Q}}((X,L_X); (X,L_Y))$ (see \eqref{eq:qgh-dist}), or the related notion
\[
\dist_{\text{GH}}((X,L_X); (X,L_Y)):= \dist_{\text{GH}}\left((\C S(X), d_{L_X}); (\C S(Y), d_{L_Y})  \right),
\] 
where the right hand side is the classical Gromov-Hausdorff distance \cite{Evans:Prob} applied to the state spaces. 
Formally speaking, $\distqgh$ and $\distgh$, as defined above, are only metrics on the space of isometry classes of compact quantum metric spaces \cite[Theorem 7.8]{Rie:GHD}, just like the ordinary Gromov-Hausdorff distance is a metric on the space of isometry classes of compact metric spaces. As is customary, we shall suppress the difference between a compact quantum metric space and its isometry class in the sequel. The main result of this paper is that these two notions of quantum Gromov-Hausdorff distance are not equivalent:
\begin{theoremletter}\label{thm:mainthm}
The two metrics $\distgh$ and $\distqgh$  are not equivalent. 
\end{theoremletter}

The strategy to prove this result  is to devise a sequence $(X_n, L_n)$ of compact quantum metric spaces which is convergent for $\distgh$ but which is bounded  away from the limit point with respect to $\distqgh$. The estimates needed in order to prove this are surprisingly subtle and involve a careful analysis of the so-called Mazur map as well as an application of Clarkson's inequalities; see Section \ref{sec:estimates} for more details. \\

\paragraph{\textbf{Acknowledgments.}}
The authors gratefully acknowledge the financial support from  the Independent Research Fund Denmark through grant no.~9040-00107B, 7014-00145B and 1026-00371B.
This note grew out of discussions with Walter van Suijlekom during the 2022 NseaG-workshop at the Lorentz center in Leiden. We would like to thank the Lorentz center for their hospitality and the nice research environment they provide.
Furthermore, the authors wish to thank Alexander Kahle, Fr\'{e}d\'{e}ric Latr\'{e}moli\`ere and Marc Rieffel  for their input in connection with the problem addressed in the present note.

\section{Compact quantum metric spaces}\label{sec:qcms}

The theory of compact quantum metric spaces has yet to reach its final form, and there are therefore a number of competing, but strongly related, definitions available \cite{Rie:MSS, Li:CQG, Lat:QGH}. We shall here use the setting of operator systems as  basis for the definitions, since this seems to provide a convenient flexibility compared with the $C^*$-algebraic setting \cite{Li:CQG, Lat:QGH}, and furthermore has the advantage of being in line with some of the recent developments in non-commutative geometry \cite{walter:GH-convergence, walter-connes:truncations}.  By a \emph{complete operator system} we shall mean a closed subset $X$ of a unital $C^*$-algebra $A$, which is furthermore invariant under the adjoint operation and contains the unit of $A$. The usual definition of a \emph{state} carries over to this setting, and the state space $\C S(X)$ becomes a compact topological space for the weak$^*$ topology.  
\begin{dfn}
A \emph{compact quantum metric space} is a complete operator system $X$ equipped with a semi-norm $L\colon X\to [0,\infty]$ satisfying that
\begin{enumerate}
\item $L$ is $*$-invariant and $L(1)=0$.
\item the set $\text{dom}(L):= \{x\in X : L(x)<\infty\}$ is dense in $X$.
\item the function $d_L\colon \C S(X) \times \C S (X) \to [0,\infty]$ given by
\[
d_{L}(\mu, \nu):=\sup\{ |\mu(x)-\nu(x)| : L(x)\leq 1 \}
\]
metrises the weak$^*$ topology. 
\end{enumerate}
In this case, $L$ is referred to as a \emph{Lip-norm}.
\end{dfn}

Since the purpose of the present note is somewhat specific, we shall not dwell on the basic properties of compact quantum metric spaces, but refer the reader to \cite[Section 2]{KK:quantumSU2}
for a comparison between the definition above and Rieffel's original definition \cite{Rie:MSS}, and the references therein for a number of important examples. 
In analogy with the classical Gromov-Hausdorff distance  \cite{gromov-groups-of-polynomial-growth-and-expanding-maps, edwards-GH-paper}, Rieffel developed his quantum Gromov-Hausdorff distance based on the notion of an {admissible seminorm}. More precisely, if $(X,L_X)$ and $(Y,L_Y)$ are compact quantum metric spaces a Lip-norm $L\colon X\oplus Y \to [0,\infty]$ is called \emph{admissible} if $\dom(L)=\dom(L_X)\oplus \dom(L_Y)$ and if the quotient semi-norm induced by $L$ on the self-adjoint parts $X_{\sa}$ and $Y_{\sa}$ agree with (the restrictions of) $L_X$ and $L_Y$, respectively. 
Given such an admissible seminorm $L$, the coordinate projections dualise to isometric embeddings
\[
(\C S(X), d_{L_X}) \overset{\iota_X}{\longhookrightarrow}  (\C S(X\oplus Y), d_L) \overset{\iota_Y}{\longhookleftarrow} (\C S(Y), d_{L_Y})
\]
and it therefore makes sense to define 
\begin{align}\label{eq:qgh-dist}
\distqgh((X,L_X);(Y,L_Y)):=\inf\big\{ \disth\big( \iota_X(\C S(X)), \iota_Y(\C S(Y))    \big) : L \text{ admissible} \big\}
\end{align}
where $\disth$ is the Hausdorff distance  \cite{Hausdorff-grundzuge} measured with respect to $d_L$.
As demonstrated in \cite[Lemma 2.2]{KK:quantumSU2}, the above quantum Gromov-Hausdorff distance agrees with Rieffel's original definition \cite{Rie:GHD} applied to the order unit spaces $\dom(L_X)_{\sa}$ and $\dom(L_Y)_{\sa}$. We may therefore also apply Li's alternative description from \cite[Proposition 3.2]{Li:GH-dist}: 
\begin{align}\label{eq:Li-formula}
\distqgh((X,L_X);(Y,L_Y)) =\inf \big\{ \disth\big(i_X(\C S(X)), i_Y(\C S(Y)) \big)\big\}
\end{align}
where the infimum runs over all  \emph{affine} isometric embeddings $i_X\colon \C S(X) \to V$ and  $i_Y\colon \C S(Y) \to V$ into arbitrary real normed spaces $V$ and $\disth$ denotes the Hausdorff distance measured  with respect to the metric coming from the norm on $V$.
 It follows immediately from \eqref{eq:qgh-dist} that $\dist_{\text{GH}}^{\text{Q}}$ dominates  $\dist_{\text{GH}}$ and to prove Theorem \ref{thm:mainthm}, we shall therefore show that no constant $C>0$ exists for which $ \distqgh \leq C \cdot \distgh$. To do so, we have to manufacture compact quantum metric spaces with quite specific state spaces which is possible due to another result of Li  \cite[Proposition 3.1]{Li:GH-dist}.  This result ensures that any  compact convex subset $K$ of a  real normed space $V$ is affinely isometrically isomorphic  to the state space of a compact quantum metric space. In our setting, the relevant operator system $X$ is the closure of the affine Lipschitz continuous functions $\text{Lip}_{\text{Aff}}(K, \cc)$ inside the unital $C^*$-algebra $C(K, \cc)$ and the corresponding Lip-norm computes the Lipschitz constant of the functions in $X$.
%

 We are going to apply the above constructions to the unit balls in finite dimensional $\ell^p$-spaces. More precisely, for $p\geq 1$ and $N\in \nn$ we consider the compact convex subset
\[
K_p^N := \{ x\in \rr^N : \|x\|_p\leq 1\}  \subseteq \RR^N ,
\]
where we consider $\RR^N$ as a real normed space with respect to the $p$-norm, so that the metric on the closed unit ball $K_p^N$ is given by $d_p^N(x,y):=\|x-y\|_p$. We denote by $(X_p^N, L_p^N)$ the associated compact quantum metric space whose state space $(\C S(X_p^N), d_{L_p^N})$ is affinely, isometrically isomorphic to $(K_p^N, d_{p}^N)$.

  In the following section,  we show that there exist a sequence $\{p_n\}$ in $(1,2)$ with $\lim_{n\to \infty} p_n= 1$ and a sequence $\{N_n\}$ of natural numbers such that
\[
\lim_{n \to \infty} \distgh\left( (X_{p_n}^{N_n},L_{p_n}^{N_n}); (X_1^{N_n}, L_1^{N_n}) \right) = 0
\]
but for which
\[
\distqgh\left( (X_{p_n}^{N_n},L_{p_n}^{N_n}); (X_1^{N_n}, L_1^{N_n}) \right) \geq \frac{1}{4} 
\]
for all $n \in \nn$. This implies the statement of Theorem \ref{thm:mainthm}.

We remark that a similar problem was solved within the theory of Banach spaces in \cite{KO:distances-between-Banach-spaces} where the authors compare a version of the Gromov-Hausdorff distance to the so-called Kadets distance, and the constructions in the present paper draw also on the techniques from \cite{KO:distances-between-Banach-spaces}.


\section{Estimating the distances}\label{sec:estimates}
The first step towards the proof of Theorem \ref{thm:mainthm} is to obtain an estimate on the Gromov-Hausdorff distance between $X_p^N$ and $X_1^N$ which is independent of $N$. We provide such an estimate in the following proposition, whose proof is  a minor modification of the corresponding argument for infinite dimensional $\ell^p$-spaces given in \cite{KO:distances-between-Banach-spaces};  we include it here, with added details,  for the reader's convenience.

\begin{prop}\label{prop:GH-convergence}
Let $N\in \nn$ and  $p\in [1,\infty)$. It holds that 
\[
\distgh\left((X_p^N,L_p^N); (X_1^N, L_1^N)   \right)\leq  2^p-2.
\]
\end{prop}
For the proof of Proposition \ref{prop:GH-convergence}, it is convenient to first single out the following lemma.

\begin{lemma}\label{lem:estimates}
Let $p\in [1,\infty)$. The following holds:
\begin{enumerate} 
\item The function $f\colon [0,2] \to \rr$ given by $f(x)=\left|x^p-x \right|$ attains its maximum at $x=2$.
\item For all $t\geq s\geq 0$ it holds that  $(t-s)^p + s^p \leq t^p \leq 2^{p-1}\big((t-s)^p + s^p\big)$.
\end{enumerate}
\end{lemma}
\begin{proof}
 The two claims are trivial for $p = 1$ so we focus on the case where $p \in (1,\infty)$.
  
 We start out by proving (1). The only candidates for the maximum are at the points $0,1,2$ and the critical point $x_0 := \big( \frac{1}{p} \big)^{1/(p-1)}$. Since $f$ vanishes at $0$ and $1$ we only need to show that 
\[
x_0 - x_0^p \leq 2^p-2 .
\]
To establish this inequality, put $q := p/(p-1)$ and record that
\[
x_0 - x_0^p = \frac{1}{p^{q-1}} - \frac{1}{p^q} = \frac{p-1}{p^q} \leq p - 1 \leq 2^p - 2 . 
\]

We now turn to the proof of (2). As $p>1$, the function $g\colon [0,\infty)\to \RR$ given by $g(x)=x^p$ is convex and hence
\[
(\tfrac{1}{2}x+\tfrac{1}{2}y)^p\leq \tfrac{1}{2}x^p +\tfrac{1}{2}y^p,
\]
so that $(x+y)^p\leq 2^{p-1}(x^p+y^p)$. Furthermore, as a consequence of the triangle inequality for the $p$-norm, it holds that $x^p + y^p \leq (x + y)^p$.
The two inequalities in (2) now follow by putting $x = t - s$ and $y = s$.
\end{proof}

\begin{proof}[Proof of Proposition \ref{prop:GH-convergence}]
We shall use the characterisation of Gromov-Hausdorff distance in terms of correspondences as described, for instance, in \cite{Evans:Prob}. More precisely, \cite[Theorem 4.11]{Evans:Prob} shows that if one can find a bijective map $\varphi\colon A\to B$ between compact metric spaces $(A,d_A)$ and $(B,d_B)$ whose \emph{distortion}
\[
\text{dis}(\varphi):=\sup\{|d_A(a_1,a_2)-d_B(\varphi(a_1), \varphi(a_2))| : a_1, a_2\in A\}
\]
is bounded by $\ep$, then $\distgh((A,d_A), (B,d_B))\leq  \frac{1}{2}\ep$. In our situation, the relevant map will be the \emph{Mazur map} $\varphi\colon K_p^N \to K_1^N$ given by
\[
\varphi(x_1,\dots, x_N)=\left(\sgn(x_1)|x_1|^p, \dots, \sgn(x_N) |x_N|^p\right),
\]
and we therefore need to estimate the quantity
\[
\Big|\|x-y\|_p - \|\varphi(x)-\varphi(y)\|_1 \Big| 
\]
for $x,y\in K_p^n$. We now claim that
\[
\Big| |a-b|^p - \big| \sgn(a)|a|^p-\sgn(b)|b|^p \big|  \Big| \leq (2^{p-1}-1)(|a|^p + |b|^p) \qquad (a,b\in \rr).
\]
To see this, note first that by symmetry it suffices to treat the two cases $a\geq b\geq 0$ and $b\geq 0$ and $a<0$; in the first case the claim follows from  Lemma \ref{lem:estimates} (2) 
with $t=a$ and $s=b$, and in the second case by using $s=|a|$ and $t=b+|a|$. 
For $x,y\in K_p^N$, we therefore get
\begin{align*}
\Big|\|x-y\|_p^p - \|\varphi(x)-\varphi(y)\|_1 \Big|&\leq  \sum_{i=1}^N \Big| |x_i-y_i|^p - \big| \sgn(x_i)|x_i|^p-\sgn(y_i)|y_i|^p \big|  \Big|\\
&\leq \sum_{i=1}^N (2^{p-1}-1)(|x_i|^p +|y_i|^p)\leq 2(2^{p-1}-1)=2^p-2. 
\end{align*}
Combining this with the estimate from  Lemma \ref{lem:estimates} (1) we obtain
\begin{align*}
\Big| \|x-y\|_p - \|\varphi(x)-\varphi(y)\|_1 \Big| &\leq \Big|\|x-y\|_p - \|x-y\|_p^p \Big| + \Big|\|x-y\|_p^p - \|\varphi(x)-\varphi(y)\|_1 \Big|\\
&\leq 2^p-2 +2^p-2=2(2^p-2)
\end{align*}
Thus, the bijection $\varphi$ has distortion at most $2(2^p-2)$ and hence 
\[
\distgh((X_p^N,L_p^N); (X_1^N, L_1^N))=\distgh((K_p^N, d_p^N); (K_1^N,d_1^N))\leq  2^p-2 . \qedhere
\]
\end{proof}

Next we turn towards a lower bound on the quantum Gromov-Hausdorff distance:

\begin{prop}\label{prop:lower-bound}
 Let $N\in \nn$ and $p\in (1,2)$. It holds that 
\[
\distqgh\left((X_p^N,L_p^N); (X_1^N, L_1^N)   \right)\geq \frac12- N^{\frac{1}{p}-1}.
\]
\end{prop}

For the proof we need the following two lemmas, which are likely known to experts in the field but for which we were unable to find an explicit reference.

\begin{lemma}\label{lem:signs}
  Let $p \in (1,2)$ and $N \in \nn$. For every $n\in \nn$ and $x_1, \ldots, x_n\in K_p^N$ there exist $\alpha_1, \dots, \alpha_n\in \{-1,1\}$ such that  $\| \sum_{i=1}^n \alpha_ix_i\|_p\leq n^{\tfrac{1}{p}}$.
\end{lemma}
\begin{proof}
   The proof runs by induction on $n \in \nn$. The case where $n = 1$ is clear.
Let $n \in \nn$ and suppose that $x_1,\ldots, x_{n+1} \in K_p^N$ and that $\al_1,\ldots,\al_n \in \{-1,1\}$ are chosen such that $\| \sum_{i=1}^n \alpha_ix_i\|_p\leq n^{\tfrac{1}{p}}$.
  Put $q := \frac{p}{1-p}$ so that $1/p + 1/q = 1$. From Clarkson's inequality \cite[Theorem 2]{clarkson:uniformly-convex-spaces},
  we know that
\[
\|x+y \|_p^q + \|x- y\|_p^q \leq 2 \big(\|x\|_p^p + \|y\|_p^p \big)^{\frac{q}{p}}.
\]
for all $x,y\in \RR^N$. Hence with $x := \sum_{i = 1}^n \al_i x_i$ and $y := x_{n+1}$ we get from the induction hypothesis that
\[
\|x+y \|_p^q + \|x- y\|_p^q \leq 2 ( n + 1 )^{\frac{q}{p}} .
\]
This implies that either $\|x + y\|_p^q$ or $\|x - y \|_p^q $ is bounded by $( n + 1 )^{\frac{q}{p}}$ and hence that either $\|x + y\|_p$ or $\|x - y \|_p$ is bounded by $( n + 1 )^{\frac{1}{p}}$. 
\end{proof}

The following lemma is certainly well known, but for the sake of completeness, we present the main idea of the proof.

\begin{lemma}\label{lem:extension-lem}
Let $V$ and $W$ be  real normed spaces and denote by $\B B_1^V(0)$ the closed unit ball in $V$. Then for any affine isometric map $\varphi\colon \B B_1^V(0) \to W$ there exists a  unique isometric linear map $\Phi\colon V\to W$ such that  $\varphi(x)=\Phi(x)+ \varphi(0)$ for all $x\in \B B_1^V(0)$. 
\end{lemma}
\begin{proof}
Define the affine isometric map $\psi \colon \B B_1^V(0) \to W$ by putting $\psi(x) := \varphi(x) - \varphi(0)$. The relevant map $\Phi \colon V \to W$ is then defined by $\Phi(x) := \|x\| \cd \psi\big(\frac{1}{\|x\|}x\big)$ for $x\neq 0$ and $\Phi(0) := 0$. We leave it to the reader to verify that $\Phi$ has the desired properties.
\end{proof}

%
%
%

\begin{proof}[Proof of Proposition \ref{prop:lower-bound}]
  We use Li's formula \eqref{eq:Li-formula} to estimate the quantum Gromov-Hausdorff distance. Let therefore $V$ be any  real normed space and let  $\varphi \colon K_p^N \to V$ and $\psi \colon K_1^N \to V$ be affine isometric maps. Put $\delta:= \disth\left( \varphi(K_p^N), \psi(K_1^N) \right)$, where the Hausdorff distance is measured with respect to the metric induced by the norm $\|\cdot\|_V$ on $V$.  Upon applying a global translation if necessary, we may assume that $\varphi(0)=0$.  Denote by $e_1, \dots, e_N\in \rr^N$ the standard basis and choose $x_1,\dots, x_N\in K_p^N$ such that $\|\varphi(x_i)-\psi(e_i)\|_V\leq \delta$ for all $i\in \{1, \dots, n\}$. Choose, moreover, $x_0\in K_p^N$ such that $\| \varphi(x_0) - \psi(0) \|_V \leq \de $.

 We apply the notation $\ell^p_N$ and $\ell^1_N$ for $\B R^N$ equipped with the $p$-norm and the $1$-norm, respectively. By Lemma \ref{lem:extension-lem}, we know that $\varphi$ and $\psi_0 := \psi - \psi(0)$ extend to isometric linear maps
\[
\varphi \colon \ell^p_N \to V \q \T{and} \q \psi_0 \colon \ell^1_N \to V .
\]
Furthermore, by Lemma \ref{lem:signs} we may choose signs $\al_1, \ldots, \al_N \in \{-1,1\}$ such that
\[
\Big\| \sum_{j = 1}^N \al_j (x_j - x_0) \Big\|_p \leq 2 \cd N^{\frac{1}{p}} .
\]
We now estimate as follows:
\[
\begin{split}
N & = \Big\| \sum_{j = 1}^N \al_j e_j \Big\|_1 = \Big\| \sum_{j = 1}^N \al_j \psi_0(e_j) \Big\|_V \\
& \leq \Big\| \sum_{j = 1}^N \al_j \psi_0(e_j) - \sum_{j = 1}^N \al_j \varphi(x_j) + \sum_{j = 1}^N \al_j \psi(0)  \Big\|_V + \Big\| \sum_{j = 1}^N \al_j \big( \varphi(x_0) - \psi(0) \big) \Big\|_V \\
& \hspace{0.4cm} + \Big\| \sum_{j = 1}^N \al_j \big( \varphi(x_j) - \varphi(x_0) \big) \Big\|_V \\
& \leq \sum_{j = 1}^N \big\| \psi(e_j) - \varphi(x_j) \big\|_V + N \cd \big\| \varphi(x_0) - \psi(0) \big\|_V + \Big\|  \sum_{j = 1}^N \al_j \big( \varphi(x_j) - \varphi(x_0) \big) \Big\|_V \\
& \leq 2N \de + 2 N^{1/p} .
\end{split}
\]
This establishes the desired inequality
$
\frac{1}{2} - N^{\frac{1}{p} - 1} \leq \de 
$
and the proof is complete.
\end{proof}

With the above results at our disposal, Theorem \ref{thm:mainthm} now follows easily:

\begin{proof}[Proof of Theorem \ref{thm:mainthm}]
  Let $\{p_n\}$ be a sequence in $(1,2)$ with $\lim_{n \to \infty} p_n = 1$. For each $n \in \nn$, choose $N_n \in \nn$ such that $N_n^{\tfrac{1}{p_n}-1} \leq 1/4$. By Proposition \ref{prop:lower-bound} we then get that
  \[
\distqgh\left( (X_{p_n}^{N_n}, L_{p_n}^{X_n} ); (X_{1}^{N_n}, L_{1}^{X_n} ) \right)\geq \frac{1}{4} \q \T{for all } n \in \nn .
\]
On the other hand, by Proposition \ref{prop:GH-convergence}, we have that
\[
\distgh\left((X_{p_n}^{N_n}, L_{p_n}^{N_n}); (X_1^{N_n}, L_1^{N_n}) \right) \leq 2^{p_n} - 2 \q \T{for all } n \in \nn 
\]
and hence that
\[
\lim_{n \to \infty} \distgh\left((X_{p_n}^{N_n}, L_{p_n}^{N_n}); (X_1^{N_n}, L_1^{N_n}) \right) = 0 . \qedhere
\]
\end{proof}

\begin{remark}
Theorem \ref{thm:mainthm} only shows that $\distgh$ and $\distqgh$ are inequivalent metrics, but does not exclude the fact that they could induce the same topology on the space of (isometry classes) of compact quantum metric spaces. Although we believe this not to be the case, the techniques above seem not to be adaptable to prove such a more general result. 
\end{remark}

\begin{remark}
Despite the difference between $\distqgh$ and $\distgh$ demonstrated above, van Suijlekom's main convergence result \cite[Theorem 5]{walter:GH-convergence} holds true also when $\distgh$ is replaced by $\distqgh$, as can be seen, for instance, by an application of \cite[Proposition 2.14]{KK:quantumSU2}.
\end{remark}

\bibliographystyle{plain}

\begin{thebibliography}{10}

\bibitem{clarkson:uniformly-convex-spaces}
James~A. Clarkson.
\newblock Uniformly convex spaces.
\newblock {\em Trans. Amer. Math. Soc.}, 40(3):396--414, 1936.

\bibitem{walter-connes:truncations}
Alain {Connes} and Walter~D. {van Suijlekom}.
\newblock {Spectral truncations in noncommutative geometry and operator
  systems}.
\newblock {\em {Commun. Math. Phys.}}, 383(3):2021--2067, 2021.

\bibitem{CvS:Truncations}
Alain {Connes} and Walter~D. {van Suijlekom}.
\newblock {Tolerance relations and operator systems}.
\newblock {\em {Acta Sci. Math.}}, 88(1):101--129, 2022.

\bibitem{edwards-GH-paper}
David~A. Edwards.
\newblock The structure of superspace.
\newblock In {\em Studies in topology ({P}roc. {C}onf., {U}niv. {N}orth
  {C}arolina, {C}harlotte, {N}. {C}., 1974; dedicated to {M}ath. {S}ect.
  {P}olish {A}cad. {S}ci.)}, pages 121--133, 1975.

\bibitem{Evans:Prob}
Steven~N. Evans.
\newblock {\em Probability and real trees}, volume 1920 of {\em Lecture Notes
  in Mathematics}.
\newblock Springer, Berlin, 2008.
\newblock Lectures from the 35th Summer School on Probability Theory held in
  Saint-Flour, July 6--23, 2005.

\bibitem{GvS:TolFin}
Mick Gielen and Walter~D. {van Suijlekom}.
\newblock Operator systems for tolerance relations on finite sets.
\newblock {\em Indag. Math} (to appear), 2022.
\newblock \href{https://arxiv.org/abs/2207.07735v1}{arXiv:2207.07735v1}

\bibitem{gromov-groups-of-polynomial-growth-and-expanding-maps}
Mikhael Gromov.
\newblock Groups of polynomial growth and expanding maps.
\newblock {\em Inst. Hautes \'{E}tudes Sci. Publ. Math.}, (53):53--73, 1981.

\bibitem{Hausdorff-grundzuge}
Felix Hausdorff.
\newblock {\em Grundz\"{u}ge der {M}engenlehre}.
\newblock Chelsea Publishing Company, New York, N. Y., 1949.

\bibitem{KK:quantumSU2}
Jens Kaad and David Kyed.
\newblock The quantum metric structure of quantum {$SU(2)$}.
\newblock {\em Preprint}, 2022.
\newblock \href{https://arxiv.org/abs/arXiv:2205.06043}{arXiv:2205.06043}


\bibitem{KO:distances-between-Banach-spaces}
Nigel~J. Kalton and Mikhail~I. Ostrovskii.
\newblock Distances between {B}anach spaces.
\newblock {\em Forum Math.}, 11(1):17--48, 1999.

\bibitem{Ker:MQG}
David Kerr.
\newblock Matricial quantum {G}romov-{H}ausdorff distance.
\newblock {\em J. Funct. Anal.}, 205(1):132--167, 2003.

\bibitem{Kerr-Li}
David Kerr and Hanfeng Li.
\newblock On {G}romov-{H}ausdorff convergence for operator metric spaces.
\newblock {\em J. Operator Theory}, 62(1):83--109, 2009.

\bibitem{Lat:DGH}
Fr\'{e}d\'{e}ric Latr\'{e}moli\`ere.
\newblock The dual {G}romov-{H}ausdorff propinquity.
\newblock {\em J. Math. Pures Appl. (9)}, 103(2):303--351, 2015.

\bibitem{Lat:QGH}
Fr\'{e}d\'{e}ric Latr\'{e}moli\`ere.
\newblock The quantum {G}romov-{H}ausdorff propinquity.
\newblock {\em Trans. Amer. Math. Soc.}, 368(1):365--411, 2016.

\bibitem{Lat:MGP}
Fr\'{e}d\'{e}ric Latr\'{e}moli\`ere.
\newblock The modular {G}romov-{H}ausdorff propinquity.
\newblock {\em Dissertationes Math.}, 544:70, 2019.

\bibitem{Lat:GPS}
Fr\'{e}d\'{e}ric Latr\'{e}moli\`ere.
\newblock The {G}romov-{H}ausdorff propinquity for metric spectral triples.
\newblock {\em Adv. Math.}, 404:Paper No. 108393, 2022.

\bibitem{Li:CQG}
Hanfeng Li.
\newblock {$C^*$}-algebraic quantum {G}romov--{H}ausdorff distance, 2003.

\bibitem{Li:GH-dist}
Hanfeng Li.
\newblock Order-unit quantum {G}romov-{H}ausdorff distance.
\newblock {\em J. Funct. Anal.}, 231(2):312--360, 2006.

\bibitem{Rie:MSS}
Marc~A. Rieffel.
\newblock Metrics on state spaces.
\newblock {\em Doc. Math.}, 4:559--600, 1999.

\bibitem{Rie:GHD}
Marc~A. Rieffel.
\newblock Gromov-{H}ausdorff distance for quantum metric spaces.
\newblock {\em Mem. Amer. Math. Soc.}, 168(796):1--65, 2004.
\newblock Appendix 1 by Hanfeng Li.

\bibitem{walter:GH-convergence}
Walter~D. {van Suijlekom}.
\newblock {Gromov-Hausdorff convergence of state spaces for spectral
  truncations}.
\newblock {\em {J. Geom. Phys.}}, 162:11, 2021.
\newblock Id/No 104075.

\end{thebibliography}

\end{document}